\documentclass[11pt,leqno]{article}
\usepackage{graphicx, amsfonts, amsthm, amsxtra, amssymb, verbatim, makeidx}
\usepackage{subeqnarray, relsize}
\usepackage[mathscr]{euscript}
\usepackage{hyperref}
\makeindex
\makeglossary
\begin{document}
\newtheorem{theo}{Theorem}
\newtheorem{exam}{Example}
\newtheorem{coro}{Corollary}
\newtheorem{defi}{Definition}
\newtheorem{prob}{Problem}
\newtheorem{lemm}{Lemma}
\newtheorem{prop}{Proposition}
\newtheorem{rem}{Remark}
\newtheorem{conj}{Conjecture}
\newtheorem{calc}{}

\def\group{{\sf G}} 
\def\pg{{ \sf S}}               
\def\TS{{\bf T}}
\def\NB{{\bf N}}

\def\Z{\mathbb{Z}}                   
\def\Q{\mathbb{Q}}                   
\def\C{\mathbb{C}}                   
\def\N{\mathbb{N}}                   
\def\uhp{{\mathbb H}}                
\def\A{\mathbb{A}}                   
\def\dR{{\rm dR}}                    
\def\F{{\cal F}}                     
\def\Sp{{\rm Sp}}                    
\def\Gm{\mathbb{G}_m}                 
\def\Ga{\mathbb{G}_a}                 
\def\Tr{{\rm Tr}}                      
\def\tr{{{\mathsf t}{\mathsf r}}}                 
\def\spec{{\rm Spec}}            
\def\ker{{\rm ker}}              
\def\GL{{\rm GL}}                
\def\ker{{\rm ker}}              
\def\coker{{\rm coker}}          
\def\im{{\rm Im}}               
\def\coim{{\rm Coim}}            
\def\p{{\sf  p}}
\def\U{{\cal U}}   
\def\k{{\sf k}}                     
\def\ring{{\sf R}}                   
\def\X{{\sf X}}                      
\def\T{{\sf T}}                      
\def\Ts{{\sf S}}
\def\cmv{{\sf M}}                    
\def\BG{{\sf G}}                       
\def\podu{{\sf pd}}                   
\def\ped{{\sf U}}                    
\def\per{{\bf  P}}                   
\def\gm{{\sf  A}}                    
\def\gma{{\sf  B}}                   
\def\ben{{\sf b}}                    

\def\Rav{{\mathfrak M }}                     
\def\Ram{{\mathfrak C}}                     
\def\Rap{{\mathfrak G}}                     

\def\nov{{\sf  n}}                    
\def\mov{{\sf  m}}                    
\def\Yuk{{\sf C}}                     
\def\Ra{{\sf R}}                      
\def\hn{{\sf h}}                      
\def\cpe{{\sf C}}                     
\def\g{{\sf g}}                       
\def\t{{\sf t}}                       
\def\pedo{{\sf  \Pi}}                  

\def\Der{{\rm Der}}                   
\def\MMF{{\sf MF}}                    
\def\codim{{\rm codim}}                
\def\dim{{\rm    dim}}                
\def\Lie{{\rm Lie}}                   
\def\gg{{\mathfrak g}}                

\def\u{{\sf u}}                       

\def\imh{{  \Psi}}                 
\def\imc{{  \Phi }}                  
\def\stab{{\rm Stab }}               
\def\Vec{{\rm Vec}}                 
\def\prim{{\rm  0}}                  

\def\Fg{{\sf F}}     
\def\hol{{\rm hol}}  
\def\non{{\rm non}}  
\def\alg{{\rm alg}}  
\def\tra{{\rm tra}}  

\def\bcov{{\rm \O_\T}}       

\def\leaves{{\cal L}}        

\def\cat{{\cal A}}              
\def\im{{\rm Im}}               

\def\pn{{\sf p}}              
\def\Pic{{\rm Pic}}           
\def\free{{\rm free}}         
\def \NS{{\rm NS}}    
\def\tor{{\rm tor}}
\def\codmod{{\nu}}    

\def\GM{{\rm GM}}

\def\perr{{\sf q}}        
\def\perdo{{\cal K}}   
\def\sfl{{\mathrm F}} 
\def\sp{{\mathbb S}}  

\newcommand\diff[1]{\frac{d #1}{dz}} 
\def\End{{\rm End}}              

\def\sing{{\rm Sing}}            
\def\cha{{\rm char}}             
\def\Gal{{\rm Gal}}              
\def\jacob{{\rm jacob}}          
\def\tjurina{{\rm tjurina}}      
\newcommand\Pn[1]{\mathbb{P}^{#1}}   
\def\Ff{\mathbb{F}}                  

\def\O{{\cal O}}                     
\def\as{\mathbb{U}}                  
\def\ring{{\mathsf R}}                         
\def\R{\mathbb{R}}                   

\newcommand\ep[1]{e^{\frac{2\pi i}{#1}}}
\newcommand\HH[2]{H^{#2}(#1)}        
\def\Mat{{\rm Mat}}              
\newcommand{\mat}[4]{
     \begin{bmatrix}
            #1 & #2 \\
            #3 & #4
       \end{bmatrix}
    }                                

\newcommand\SL[2]{{\rm SL}(#1, #2)}    
\def\gcd{{\rm gcd}}                  

\def\cf{r}   
\def\per{{\sf pm}}

\def\Im{{\rm  Im}}
\def\Re{{\rm  Re}}

\begin{center}
{\LARGE\bf On elliptic modular foliations, II 
}
\\
\vspace{.25in} {\large {\sc Hossein Movasati}}
\footnote{
Instituto de Matem\'atica Pura e Aplicada, IMPA, Estrada Dona Castorina, 110, 22460-320, Rio de Janeiro, RJ, Brazil,
{\tt www.impa.br/$\sim$ hossein, hossein@impa.br}}
\end{center}

\begin{abstract}
We give an example of a one dimensional foliation $\F$ of degree two in a Zariski open set of a four dimensional weighted 
 projective space  which  has only an enumerable set  of algebraic leaves. These  are defined over rational numbers  and are isomorphic  to modular curves  $X_0(d),\ d\in\N$ minus cusp points. 
As a by-product we get  new models for modular curves for which we slightly modify an argument  due to J. V. Pereira and give  closed formulas for elements in 
 their defining ideals. The general belief has been that such formulas do not exist
 and the emphasis in the literature has been on introducing faster algorithms to compute equations for small values of $d$. 
\end{abstract}

\section{Introduction}
One of the central challenges in the theory of holomorphic foliations in compact complex manifolds is to give a criterion or an algorithm which tells us
when a foliation has a first integral. G. Darboux is the founding father of this problem who in \cite{Darboux1978} proves that if a 
foliation in $\Pn 2$ has an infinite number of algebraic leaves then it has a first integral, and hence all its leaves are algebraic. Motivated by Darboux's result 
H. Poincar\'e in \cite{Poincare1897} obtained partial results in this direction assuming  certain type of singularities for the foliation. 
J. P. Jouanolou in \cite{jouanolou79} generalizes  Darboux's result, with some extra conditions, to codimension one foliations, and 
E. Ghys in \cite{ghys00} remarks that such conditions are not necessary.
The spirit of all these results are that of  Darboux: each algebraic leaf
induces an element in a finite dimensional vector space, and the existence of infinite number of such leaves result in many linear relations among such
elements, which will eventually give us the first integral. Similar ideas can be applied for foliation with infinite number of invariant hypersurfaces,
see \cite{correa10}.  The main topic of many recent works has been to find criterion for Darboux integrability of a foliation, 
see \cite{bost01, Pereira2016} 
and the references therein.
Bounding the degree and genus of leaves of holomorphic foliations of fixed degree is another challenge which  is
mainly attributed to P. Painlev\'e and it has attracted many research, see \cite{Ollagnier2001, Pereira2016} and in particular the introduction of 
\cite{li04, LinsNeto2002} for the history 
of this problem. For an excellent expository text on both topics the reader is referred to J. V. Pereira's monograph \cite{Pereira2003}. 

Darboux's theorem in dimensions greater than $3$ is false and the first counterexample is the  
E. Picard's ``\'equation diff\'erentielle curieuse" in
\cite{Picard1889} pages 298-299. 
For a moduli space interpretation of Picard's differential equation see \S\ref{BolsoHaddad2018}. 
Picard uses  an evaluation of the Weierstrass $\wp$ function to give 
an infinite number of algebraic solutions to the Painlev\'e VI equation with a particular parameter, see also Corollary 2 and 3 in \cite{Loray2016}. 
The next class of examples is due to A. Lins Neto in \cite{li04, LinsNeto2002} in which he construct one dimensional foliations in $\Pn 2\times \Pn 1$ with a first integral which 
is the projection in $\Pn 1$ and for an enumerable subset $E$ of $\Pn 1$ the foliations in $\Pn 2 \times \{t\},\  t\in \Pn 1$ has a first integral 
of increasing degree and genus of fibers. For counterexamples in non-algebraic context see \cite{ghys00, Winkelmann2004}. 
There is no classification of foliations in higher dimensions
without a first integral and with only an enumerable set of algebraic leaves with increasing degree and genus.   In this article 
inspired by the concept of join of polynomials in singularity theory, see \cite{arn}, we introduce the self-join of foliations and for one 
example   we show that it gives us a  new counterexample to Darboux's integrability and the problem of bounding the degree and genus of algebraic leaves.
The main ingredient is the geometric interpretation of Ramanujan's differential equation using Gauss-Manin connection of elliptic curves which appears
first in N. Katz's article \cite{ka73} and is fully elaborated in \cite{ho06-3}, see also the lecture notes \cite{ho14}.

\def\v{{\rm v}}  
Let $\F(\v)$ be the  foliation in $(x_2,x_3,y_2,y_3)\in \C^4$ given by:
\begin{align}
\label{14March2018}
\v:=&\left(2x_2-6x_3+\frac{1}{6}(x_2-y_2)x_2\right)\frac{\partial}{\partial x_2}+
 \left(3x_3-\frac{1}{3}x_2^2+\frac{1}{4}(x_2-y_2)x_3\right)\frac{\partial}{\partial x_3}\\
 -&\left(2y_2-6y_3+\frac{1}{6}(y_2-x_2)y_2\right)\frac{\partial}{\partial y_2}-
 \left(3y_3-\frac{1}{3}y_2^2+\frac{1}{4}(y_2-x_2)y_3\right)\frac{\partial}{\partial y_3}.\nonumber
\end{align}
There is no $x_1$ variable and the indices for $x_i$ and $y_i$ are chosen because of their natural weights.
The singular set of the foliation $\F(\v)$  in the weighted projective space 
$\Pn w:=\Pn {2,3,2,3,1} $ 
with the coordinate 
system $[x_2:x_3:y_2:y_3:y_1]$ consists of an isolated point and a rational curve:
\begin{equation}
 \sing(\F(\v))=\{0\}\cup \left \{27x_3^2-x_2^3=27y_3^2-y_2^3= x_3^{\frac{1}{3}}+y_3^{\frac{1}{3}}-2y_1=0 \right \}.
\end{equation}
We parametrize the curve singularity of $\sing(\F(\v))$  by 
\begin{equation}
g: \Pn 1\to \Pn w,\ \ [t:s]\mapsto \left[3t^2:t^3: 3s^2: s^3: \frac{1}{2}(s+t)\right ]. 
\end{equation}
It can be easily checked that the following union of two hypersurfaces 
\begin{align}
& \Delta:=\Delta_1\cup \Delta_2,\\
& \Delta_1:=\{27x_3^2-x_2^3=0\},\ \Delta_2:=\{27y_3^2-y_2^3=0\}
\end{align}
is tangent to $\F(\v)$. 
Recall the modular curve 
\begin{equation}
X_0(d):= \Gamma_0(d)\backslash \uhp^*,\  \ \ d\in\N, 
\end{equation}
where $\uhp^*:=\uhp\cup\Q\cup\{\infty\}$ and 
$$
\Gamma_0(d):=\left\{\mat{a_1}{a_2}{a_3}{a_4}\in \SL 2\Z\Bigg| a_3\equiv 0\ \ ({\rm mod }\ \  d)\right\}.
$$
It has a structure of a compact Riemann surface and points of $\Gamma_0(d)\backslash (\Q\cup \{\infty\})$ are called its cusps. 
\begin{theo}
\label{jorgeduque2018}
The foliation $\F(\v)$ has the following properties:
\begin{enumerate}
\item\label{111}
For each $d$ there is an algebraic curve $S_0(d)$ (defined over $\Q$) in $\Pn w$, not contained in $\Delta$ and
tangent to the vector field $\v$ in \eqref{14March2018} such that $S_0(d)\backslash\Delta$  
is isomorphic  to $X_0(d)$ mines cusps. Moreover, these are the only algebraic leaves
of $\F(\v)$ in the complement of $\Delta$ in $\Pn w$.
The curve $S_0(d)$ intersects $\Delta$ only at the points $g([a:-b]),\ d=ab, \ a,b\in\N$ in the curve singularity of $\F(\v)$.

\item\label{222}
There is an $\F(\v)$-invariant set $M_\R\subset \Pn w$  such that $M_\R\cap \Delta=g(\Pn 1_\R)$, and $M_\R\cap (\Pn w\backslash \Delta)$  
is a real analytic variety of dimension $5$. It contains all $S_0(d)$ and 
\begin{equation}
M_\Q:=\cup_{d=1}^\infty S_0(d)
\end{equation}
is dense  in $M_\R$. Moreover, the foliation $\F(\v)$  has a real first integral $B:\Pn w\backslash \Delta\to\R $ and $M_\R$ is inside $B^{-1}(1)$.

\end{enumerate}
\end{theo}
The topic of computing equations for  modular curves has also a long history with fruitful applications in number theory, 
see \cite{yui78, Galbraith1996, Yang2006}  
and the references therein.
It seems to the author that no expert in this area believe on  the existence of models for modular curves over $\Q$  with  closed formulas 
for its defining equations and therefore,  the emphasis has been on computing them using $q$-expansion of modular forms. 
We can also give an explicit description of $S_0(d)$ in terms of modular forms 
for $\Gamma_0(d)$,  see the map described in \eqref{4apr2018}. However, this is not the main focus of the present text. 
The most surprising aspect of the vector field $\v$  in \eqref{14March2018} is that it gives closed formulas for equations for $S_0(d)$ for arbitrary $d$. This is 
as follows. Let  $\psi(d):=d\prod_{p}(1+\frac{1}{p})$ be the Dedekind 
$\psi$ function, where $p$ runs through primes $p$ dividing $d$, and for $i=1,2,3$ let 
$\alpha_{i,j}, j=1,2,\ldots,m_{d,i} $  be the set of  monomials: 
\begin{equation}
\label{8july2011}
y_i^{a_1}x_2^{a_2}x_3^{a_3},\ i\cdot \psi(d)=ia_1+2a_2+3a_3,\ a_1,a_2,a_3\in\N_0.
\end{equation}
For $i=1$ let us consider such monomials with $y_1=1$. 
We consider $\v$ as a derivation from  $\Q[x_2,x_3,y_2,y_3]$ to itself and define the matrix : 
\begin{equation}
\label{17.05.2018}
B_{d,i}(x_2,x_3,y_2,y_3)=
\begin{bmatrix}
 \alpha_{i,1}&\alpha_{i,2}&\cdots&\alpha_{i,m_{d,i}}\\
 \v(\alpha_{i,1})&\v(\alpha_{i,2})&\cdots&\v(\alpha_{i,m_{d,i}})\\
 \vdots &\vdots&\ddots&\vdots\\
 \v^{m_{d,i}-1}(\alpha_{i,1})&\v^{m_{d,i}-1}(\alpha_{i,2})&\cdots&\v^{m_{d,i}-1}(\alpha_{i, m_{d,i}})\\
\end{bmatrix}.
\end{equation}
The entries of $B_{d,i}$ are polynomials in $\Z[\frac{1}{6}][ x_2,x_3,y_2,y_3]$.
\begin{theo}
\label{16may2018}
The  polynomials  $J_{d,i}:=\det B_{d,i}, \ \ i=1,2,3$ restricted to $S_0(d)$ vanish. 
\end{theo}
The polynomial $J_{d,i}$  is huge, and it might be difficult or impossible 
to use $J_{d,i}$ for computational purposes. However, we can also  give defining equations for $S_0(d)$ with some conditions on an isogeny which can be verified computationally. 
This is as follows.
Using  Hecke operators, one can see that $S_0(d)$ in $\Pn w\backslash \Delta$ is a complete intersection given by three polynomials of the form 
\begin{equation}
\label{21may2018}
Q_{d,i}:=\sum_{j=1}^{m_{d,i}}  c_{i,j}\alpha_{i,j},\ i=1,2,3,\  c_{i,j}\in\Q
\end{equation}
and further we can assume that the coefficient of $y_i^{\psi(d)}$ in $Q_{d,i}$ is $1$. 
Let $E_1: y^2=4x^3-t_2x-t_3$ and $E_2: y^2=4x^3-s_2x-s_3$ be two elliptic curves in the Weierstrass format and let 
$f: E_1\to E_2$ be an isogeny with cyclic kernel of order $d$, all defined over $\C$. We can compute the numbers $k,k'\in\C$ from the equality 
\begin{equation}
\label{omidpiano2018}
 \begin{bmatrix}
  f^*\frac{dx}{y}, &
  f^*\frac{xdx}{y}  
 \end{bmatrix}=
\begin{bmatrix}
  \frac{dx}{y}, &
  \frac{xdx}{y}  
 \end{bmatrix}
\mat{k d^{\frac{1}{2}}}{k'd^{-\frac{1}{2}}}{0}{k^{-1}d^{-\frac{1}{2}}}
\end{equation}
 where $f^*: H^1_\dR(E_2)\to H^1_\dR(E_1)$ is the induced map in de Rham cohomologies,  see for instance \cite{ked08} or \cite{ho14}. 
\begin{theo}
\label{saket2018}
We have 
\begin{enumerate}
\item 
$\left[t_2: t_3: s_2: -s_3:   k^{-1}k' \right]\in S_0(d)\backslash \Delta$ and
any point of this set is obtained in this way. 
 \item 
 Assume that for an isogeny of elliptic curves as above we have $k'\not=0$ and define 
\begin{equation}
\label{26may2018}
p:= \left (t_2 k^2{k'}^{-2}, t_3 k^3{k'}^{-3},  s_2k^2{k'}^{-2}, -s_3 k^3{k'}^{-3}\right)\in\C^4 
\end{equation}
Let $C_{d,i}$ be the $m_{d,i}\times 1$ matrix with entries $c_{i,j},\ j=1,2,\ldots,m_{d,i}$ which are coefficients of $Q_{d,i}$ in \eqref{21may2018}. It satisfies 
$$
B_{d,i}(p)C_{d,i}=0.
$$
\end{enumerate}
\end{theo}
If the matrix $B_{d,i}(p)$ has rank $m_{d,i}-1$ then Theorem 
\ref{saket2018} determines $C_{d,i}$, and hence $Q_{d,i}$, uniquely (we have already normalized one of the coefficients of $Q_{d,i}$ to $1$). This will give a simple formula for the coefficients 
of $Q_{d,i}$ in terms of the inverse of a $(m_{d,i}-1)\times (m_{d,i}-1)$ sub matrix of $B_{d,i}$. If the isogeny for constructing the point $p$ 
is defined over a number field $k$ then we may use the fact that $c_{i,j}$'s are rational numbers and invent other formulas by decomposing $B_{d,i}(p)$ in a basis of $k/\Q$. 
In  \cite{lmfdb} we can find many examples 
of isogenies which might be used for this discussion,  see also \cite{hus04} page 96 for an explicit formula of a $2$-isogeny.

My heartfelt thanks go to J. V. Pereira from whom I learned the up-to-date status of Darboux's integrability theorem and many references in this article.
Thanks also  go F. Loray who informed me about Picard's example and his interpretation of this in terms of connections. This work is prepared when I was
giving a series of lectures on my book "Modular and automorphic forms $\&$ beyond" at IMPA. My since thanks go to both the institute and the audience. I would also like
to thank A. Salehi Golsefidy regarding his comments on a question posed in \S\ref{12may2018}. Finally, I would like to thank P. Deligne 
who pointed out a wrong formulation in \S\ref{11may2018} in one of the earlier drafts of the present text.

\section{Self join of foliations}
\label{11may2018}
\def\BG{{\sf G}}
Let $\Theta$ be a $\C$-vector space generated by global vector fields in a complex manifold $\T$ and assume that it is closed under Lie bracket.
Let also $\F=\F(\Theta)$ be the induced holomorphic foliation in $\T$.
For any vector field $\v\in\Theta$ we attach two vector fields 
$v_i,\ i=1,2$ in $\T\times \T$. For instance, $v_1$ is uniquely determined by the fact that it is tangent to $\T\times \{x\}, x\in \T$ and under  
$\T\times \{x\}\cong \T$ it is identified with $v$.
The self join $\F+\F$
of $\F$ is a foliation in $\T\times \T$ which is given by vector fields $v_1+v_2$ for all $v\in \Theta$. 
If $\F$ is of dimension $c$ then its self join is also of dimension $c$, however, note that the 
codimension of the self join $\F+\F$ is twice the codimension of $\F$. 
The self join leaves the diagonal of $\T\times \T$ invariant. The concept of self join is inspired from a similar definition in singularity theory, see \cite{arn}. 
A self join of a foliation with trivial dynamics
might have complicated dynamics. In this text we will consider self join of foliations $\F$ with a left action $\bullet$ of an algebraic group $\BG$ on $\T$ and hence we can identify 
$\Lie(\BG)$ with a $\C$-vector space of global vector fields in $\T$. We assume that  $\Lie(\BG)\subset \Theta$ and hence the action of $\BG$ leaves the leaves of $\F$ invariant. 
It turns out that the diagonal action of $\BG$ on $\T\times \T$
\begin{equation}
\label{diagaction}
(\T\times \T) \times \BG\to \T\times \T,\ \ ((t,s),\g)\mapsto (t\bullet \g, s\bullet\g)
\end{equation}
acts also on each leaf of $\F+\F$, and hence, it gives us a foliation $\F^*$ in the quotient $(\T\times \T)/\BG$, which by abuse of notation we call it again 
the self join of $\F$, being clear in the context
which we mean $\F^*$ or $\F+\F$.   
Our main example for this is the foliation $\F$ in $(t_1,t_2,t_3)\in \C^3$ given by 
three vector fields
\begin{equation}
\label{ramanve}
f=-(t_1^2-\frac{1}{12}t_2)\frac{\partial}{\partial t_1}-(4t_1t_2-6t_3)\frac{\partial}{\partial t_2}-
(6t_1t_3-\frac{1}{3}t_2^2)\frac{\partial}{\partial t_3},
\end{equation}
$$
h=-6t_3\frac{\partial }{\partial t_3}-4t_2\frac{\partial }{\partial t_2}-2t_1\frac{\partial}{\partial t_1},\ e=\frac{\partial}{\partial t_1},
$$
which has just one leaf which is $\C^3\backslash\{ 27t_3^2-t_2^3=0\}$ and its complement is the singular set of $\F$.
The $\C$-vector space generated by these vector fields equipped with the classical bracket of vector fields  
is isomorphic to the Lie Algebra \(\mathfrak{sl}_2\):
\begin{equation}
\label{liealgebra}
[h,e] = 2e, \quad [h,f] = -2f, \quad [e,f] = h,
\end{equation}
see for instance \cite{Gu07} \S 3. The vector field $\Ra=-f$  is mainly attributed to Ramanujan and it is the main ingredient of the geometric theory
of quasi-modular forms, see \cite{ho06-2}  and \S\ref{10mayo2018fagulha}. Its dynamics and arithmetic properties are fairly described in \cite{ho06-3}.    
 In order to describe the self join of $\F$, we consider two copies of $\F$ in 
$(t_1,t_2,t_3)\in\C^3$  and $(s_1,s_2,s_3)\in\C^3$  with the corresponding vector fields 
$e_1,f_1,h_1$ and $e_2,f_2,h_2$, respectively.  The self join $\F+\F$ is the  foliation in $(t,s)\in \C^6$ given by $e_1+e_2, f_1+f_2,h_1+h_2$.
For this example we consider the algebraic group 
\begin{equation}
 \label{2nov10}
\BG:=
\left\{\mat{k}{k'}{0}{k^{-1}}\Bigg| \ k'\in \C, k\in \C-\{0\}\right\}
\end{equation} 
and its action on $\C^3$ is given by
$$
t\bullet \g:=(
 t_1k^{-2}+k'k^{-1},
t_2k^{-4}, t_3k^{-6}), 
$$
$$t=(t_1,t_2,t_3)\in\T,\ \ \ \ 
\g=\mat {k}{k'}{0}{k^{-1}}\in \BG.
$$
In the quotient space $\Pn w:\cong\C^3\times \C^3/\BG$ with $w=(2,3,2,3,1)$ we have the affine coordinate system 
\begin{equation}
\label{17may2018}
(x_2,x_3, y_2,y_3):= \left(\frac{t_2}{(t_1-s_1)^2}, \ \ \frac{t_3}{(t_1-s_1)^3},\ \ \frac{s_2}{(s_1-t_1)^2}, \ \ \frac{s_3}{(s_1-t_1)^3}\right).
\end{equation}
We make derivations of $x_i$ and $y_i$ variables along the vector fields and divide them over $t_1-s_1$
and conclude that the foliation $\F^*$ in the affine chart $(x_2,x_3,y_2,y_3)\in\C^4$ 
is given by the quadratic vector field  $\v$ in \eqref{14March2018}, and so, $\F^*=\F(\v)$. 
Note that in the case of $y_2$ and $y_3$ we factor out the term $s_1-t_1$
and so the corresponding terms in $\v$ have negative sign.

\section{Generalized period domain}
The generalized period domain  in the case of elliptic curves
\begin{equation}
\label{perdomain}
\pedo:= \left \{
x=\mat {x_1}{x_2}{x_3}{x_4} \Bigg| x_i\in\C,\ x_1x_4-x_2x_3=1,\  \Im(x_1\overline{x_3})>0 \right \}
\end{equation}
was first introduced in \cite{ho06-3, ho06-2}, see also \cite{ho18} for this notion in connection with  arbitrary projective varieties and 
Griffiths period domain. 
The discrete group  $\SL2\Z$ (resp. $\BG$ in (\ref{2nov10}) ) acts from the left (resp. right) on 
$\pedo$ by usual multiplication of matrices. We also call $\ped:=\SL 2\Z\backslash \pedo$ the generalized period domain. 
We consider the following family of elliptic curves
\begin{equation}
\label{khodaya-1}
E_t: y^2-4(x-t_1)^3+t_2(x-t_1)+t_3=0,\  t\in\T:= \C^3\backslash \{27t_3^2-t_2^3=0\}.
\end{equation}
with $\left[\frac{dx}{y}\right], \left[\frac{xdx}{y}\right]\in H^1_\dR(E_t)$. This is the universal
family of triples $(E,\alpha,\omega)$, where $\alpha,\omega\in H^1_\dR(E)$, $\alpha$ is holomorphic in $E$ and $\Tr(\alpha\cup \omega)=1$, for details see 
\cite{ho14} \S 5.5. 
It turns out that the
period map
$$
\per: \T\rightarrow \ped,\ t\mapsto
\left [\frac{1}{\sqrt{2\pi i}}\mat
{\int_{\delta} \frac{dx}{y}}
{\int_{\delta}\frac{xdx}{y} }
{\int_{\gamma} \frac{dx}{y}}
{\int_{\gamma}\frac{xdx}{y} } \right ].
$$
is a biholomorphism of complex manifolds and it respects the right action of $\BG$ on both sides, see \cite{ho06-3} Proposition 2.
Here, $[\cdot]$ means the equivalence class and $\{\delta,\gamma\}$ is a basis of the $\Z$-module 
$H_1(E_t,\Z)$ with $\langle \delta,\gamma\rangle=-1$. From now one we identify $\T$ with $\ped$ under this map and use
the same notations in both sides. The push-forward of the vector fields $f,e,h$ by the period map are respectively given by 
\begin{align}
\label{nossoflamengo2018}
 f  &= x_2\frac{\partial}{\partial x_1}+ x_4\frac{\partial}{\partial x_3},\\ 
 e  &=   x_1\frac{\partial}{\partial x_2}+ x_3\frac{\partial}{\partial x_4},\\
 h  &=  x_1\frac{\partial}{\partial x_1}- x_2\frac{\partial}{\partial x_2}+x_3\frac{\partial}{\partial x_3}-x_4\frac{\partial}{\partial x_4}.
\end{align} 
This follows from the computation of the Gauss-Manin connection of three parameter family of elliptic curves $E_t$,
see for instance \cite{sas74} or \cite{ho06-3} Proposition 1.   
It turns out that in order to study the self joins $\F+\F$ and $\F^*$ in the algebraic side $\T\times\T$, it is enough to study it in the holomorphic 
side $\ped\times \ped$, for which
we use the coordinate system $(x,y)\in \pedo\times\pedo$.
It can be easily checked that the local first integral 
of the foliation $\F+\F$ constructed in \S\ref{11may2018} is the function
\begin{eqnarray}
\label{div2018}
 & &F:   \ped\times\ped \to \Mat(2\times 2,\C),\\
 & & \ F(x,y):=   y\cdot x^{-1}. 
 \nonumber
\end{eqnarray} 
Note that $F$ is a multi-valued function. 
We have an action of $\SL 2\Z$ on $\pedo$. This means that analytic continuations of $F$ will result in the multiplications of $F$ both
from the left and right with elements of $\SL 2\Z$.
\begin{prop}
 \label{fazendinha2018}
 The space of leaves of $\F+\F$ is 
 \begin{equation}
 \label{deligne2018}
\SL 2\Z\mathlarger{\mathlarger{\backslash}} \left(  \SL 2\C- \left\{A\in \SL 2\C \Big|   \Re(A)=0\right \}\right)\mathlarger{\mathlarger{/}}\SL 2\Z. 
 \end{equation}
\end{prop}
\begin{proof}
The space of leaves of the foliation $\F+\F$ is defined to be 
$\SL 2\Z\backslash \Im(F)/\SL 2\Z$ and so we have to determine $\Im(F)$. We have 
$$
\BG\times \uhp\cong\pedo,\ \ (\g,\tau)\mapsto \mat{\tau}{-1}{1}{0}\cdot\g,
$$
where $\uhp$ is the upper half plane. 
We use this for $x$ and $y$ and we have  
\begin{eqnarray*}
  y\cdot x^{-1} 
  &=& \mat{a^{-1}-b\tau_2}{ -a^{-1}\tau_1+a\tau_2+b\tau_1\tau_2}{-b}{a+b\tau_1}
\end{eqnarray*}
where $(a,b,\tau_1,\tau_2)$ varies in the set $(\C-\{0\})\times \C\times\uhp \times\uhp$. For $b=0$, $(a^{-1},-a, -a^{-1}\tau_1+a\tau_2)$ only avoids
pure imaginary vectors. For $b\not=0$, the vector 
$(b, a^{-1}-b\tau_2, -a-b\tau_1)$ also avoids only pure imaginary vectors. 
Note that in this case we 
can neglect the $(2,2)$-entry of this matrix as its determinant is $1$.
\end{proof}
\begin{rem}\rm
 If we replace the period maps $x:=\per(t)$ and $y:=\per(s)$ in the definition of $F$, we conclude that there are  three quadratic combination
 of elliptic integrals which are constant along the leaves of the foliation given by the vector field $\v$. 
 This is the main reason for the naming  ``elliptic modular foliation" used in this article and \cite{ho06-3}. 
 Modular mainly refers to either that such foliations live in moduli spaces or
 they have solutions in term of modular forms. 
\end{rem}

\section{Real first integral and Levi flat}
We have the following  global real first integral  for $\F+\F$:
\begin{eqnarray*}
 & & B: \ped\times\ped\to \R,\\
 & & B(x,y):=\frac{1}{2}\Tr(F\bar F^{-1})=
 \frac{1}{2} \Tr\left( y\cdot x^{-1}\cdot \bar x\cdot \bar y^{-1}  \right).
\end{eqnarray*}
Note that by taking $F\bar F^{-1}$, the right action of $\SL 2\Z$ on $F$ is killed, and by taking the trace the left action 
(recall that $\Tr(AB)=\Tr(BA)$ for two matrices $A$ and $B$). We have also the pencil of real surfaces $\Re(F(x))+a\Im(F(x))=0,\ \ a\in\R\cup \{\infty\}$ which are invariant under both left
and right action of $\SL  2\Z$. Since $\det(F)=1$, only for  $a=0,\infty$ we might get non-empty set. By Proposition \ref{fazendinha2018}
this set for $a=0$ is also empty. Therefore, we define
\begin{eqnarray}
\label{desejos2018}
 M_{\R} & :=& \left \{ p\in \ped\times\ped \Big| \Im(F(p))=0\right \}
\end{eqnarray}
which is called a Levi-flat, see \cite{Alcides2011} for the general definition.
Note that  this set is inside the fiber $B^{-1}(1)$ of $B$ over $1$. 
We also define
\begin{eqnarray*}
M_\Q &:=& \left\{ p\in\ped\times\ped\Big | F(p)\in \sqrt{\Q^+ } \GL^+ (2,\Q)\right\}\\
     & = &  \mathlarger{\cup}_{d=1}^\infty V_d,\ \ \ \
\end{eqnarray*}
where 
\begin{equation}
\label{zoorekhar2018}
  V_d:= \left\{ p\in\ped\times\ped\Big | F(p)=A_d \right\},\ \ A_d:= \mat{ d^{\frac{1}{2}}}{0}{0}{ d^{-\frac{1}{2}}}.  
\end{equation}
The second equality is modulo the actions of $\SL 2\Z$ from the left and right on $F$. Note that 
$\SL 2\Z\backslash \Mat_d(2,\Z)/\SL 2\Z$ is naturally isomorphic to 
the set (up to isomorphism) of finite abelian groups of order $d$ and generated by at most two elements, and so, each element in this
quotient has a unique representative of the form $\mat{ d_1d_2}{0}{0}{d_2}$. 
We have the following immersion of the three dimensional non-compact version of modular curves inside $\ped\times\ped$:
\begin{equation}
\Gamma_0(d)\backslash\pedo \hookrightarrow \ped\times\ped ,\ x\mapsto \left( x ,  A_d\cdot x  \right)   
\end{equation}
whose image is exactly $V_d$ defined above and it is a connected fiber of $F$, and hence a leaf of $\F+\F$.  
\begin{prop}
\label{twoyearsago}
The only closed leaves  of $\F+\F$ in $\ped\times\ped$ are $V_d$'s for $d\in\N$.  
\end{prop}
\begin{proof}
For $A$ in the double quotient \eqref{deligne2018} the leaf $F^{-1}(A)$ is closed in $\ped\times \ped$ if and only if
the double quotient 
\begin{equation}
\label{28.03.2018}
\# \SL 2\Z \backslash \left ( \SL 2\Z\cdot A\cdot   \SL 2\Z\right) <\infty
\end{equation}
is finite. This happens if and only if the matrix $A$,  up to multiplication with a constant,   has rational entries.
Since $\det(A)=1$, we conclude that such a leaf is in $M_\Q$.
\end{proof}
\begin{rem}\rm
 \label{edemgurilandiaflamengo2018}
In the algebraic side $\T\times\T$, it is easy to verify that the affine variety $V_d\subset \T\times\T$ is the locus of isogenies
\begin{equation}
\label{CMPDavarNabood2017}
f: E_t\to E_s,\ \ f^*\frac{dx}{y}=d^{\frac{1}{2}}\cdot \frac{dx}{y},\ \  f^*\frac{xdx}{y}=d^{-\frac{1}{2}}\cdot \frac{xdx}{y}
\end{equation}
Here, $f^*: H^1_\dR(E_s)\to H^1_\dR(E_t)$ is the map induced in 
de Rham cohomologies.  
\end{rem}

\section{Eisenstein series and Halphen property}
\label{10mayo2018fagulha}
Recall the Eisenstein series:
\begin{equation}
\label{eisenstein}
E_{2k}(\tau)=1+(-1)^k\frac{4k}{B_k}\mathlarger{\sum}_{n\geq
1}\sigma_{2k-1}(n)q^{n},\ \  k=1,2,3, \ \tau\in\uhp,
\end{equation}
where $q=e^{2\pi i \tau}$ and $\sigma_i(n):=\sum_{d\mid n}d^i$ and $B_i$'s are the Bernoulli numbers, 
$B_1=\frac{1}{6},\ B_2=\frac{1}{30},\ B_3=\frac{1}{42},\ \ldots$.
S. Ramanujan in \cite{ra16} page 181 proved that 
\begin{equation}
\label{emiliomestrado2018}
g=(g_1,g_2,g_3)=(a_1E_2,a_2E_4,a_3E_6),
 \end{equation}
with $(a_1,a_2,a_3)=(\frac{2\pi i}{12},12(\frac{2\pi i}{12})^2 , 8(\frac{2\pi i}{12})^3)$ is a solution of the vector
field $\Ra:=-f$ in \eqref{ramanve} with derivation with respect to $\tau$, and so $f$ is mainly known is Ramanujan relation (or differential equation) 
between Eisenstein series. 
Thirty years before Ramanujan,  G. Halphen calculated the Ramanujan differential equation and apparently he did not know about Eisenstein series (see \cite{hal00} page 331). 
He had even a generalization of this differential equation into a three parameter family that we will discuss it in \S\ref{12may2018}. A fundamental but simple observation
due to Halphen is the following.  For a holomorphic function defined in $\uhp$ and  $A=\mat{a}{b}{c}{d}\in \SL 2\C$,   $m\in\N$ let
\begin{align}\label{marizek2018}
(f\mid_m^0A)(\tau) &:= (c\tau+d)^{-m}f(A\tau),\\ 
 (f\mid_m^1A)(\tau) &:=  (c\tau+d)^{-m}f(A\tau)-c(c\tau+d)^{-1},\nonumber
\end{align}
If $\phi_i,\ i=1,2,3$  are the coordinates of 
a solution of the Ramanujan differential equation  $\Ra$ then 
$\phi_1\mid_2^1A,\ \phi_2\mid_4^0A,\ \phi\mid_6^0A$
are also coordinates of a solution of $\Ra$  for all
$A\in\SL 2\C$. This is known as the  Halphen property. The subgroup
of $\SL 2\C$ which fixes the solution given by Eisenstein series
is $\SL 2\Z$.

\begin{rem}\rm
By Halphen property  and the construction of the vector field $\v$ in \S\ref{11may2018} it follows that 
a general solution of  this vector field in $\Pn w-\{\Delta=0\}$, up to change of coordinate system in $\tau$,   is given by
\begin{eqnarray}
\label{4apr2018-2}
& &\uhp\to \Pn w, \\  
& &\tau\mapsto \left[g_2(\tau ):g_3(\tau): (g_2|^0_4A)(\tau): -(g_3|^0_6A)(\tau): g_1(\tau)-(g_1|^1_2A)(\tau)\right ].
\end{eqnarray}
Note that in the final step of the construction of the vector field  $\v$ we 
have the  division over  $t_1-s_1$. This implies that the map \eqref{4apr2018} is not necessarily a solution of the vector field and it is only tangent 
to it.
\end{rem}

\section{Proof of Theorem \ref{jorgeduque2018}}
For the proof of Theorem \ref{jorgeduque2018} we note that $\F(\v)=\F^*:=(\F+\F)/\BG$, where $\F$ is the one leaf 
foliation in $\ped$ given by the vector fields \eqref{nossoflamengo2018} 
and we have the diagonal action of $\BG$ on $\ped\times\ped$ as in \eqref{diagaction}. 
For simplicity, we have used the same notation for objects in $\T\times \T$, $\ped\times\ped$
 or their quotients by $\BG$.  
In this and the next section we use $\Gamma:=\SL 2\Z$. 

{\it Proof of \ref{111}:} 
Let us consider the holomorphic map
\def\mm{{\mathfrak s}}
\begin{equation}
\label{4apr2018}
\mm: \ X_0(d)\to \Pn w,\ \ \ d\in\N \  
\end{equation}
which in the affine chart $\C^4\subset \Pn w$ is given by
{\small
\begin{equation}
\label{despacitopiano2018}
\ \tau\mapsto \left(\frac{g_2(\tau )}{(g_1(\tau)-d\cdot g_1(d\cdot\tau ))^2}, \ \ 
\frac{g_3(\tau)}{(g_1(\tau)-d\cdot g_1(d\cdot\tau))^3},\ \ 
\frac{d^2\cdot g_2(d\cdot \tau)}{(d\cdot g_1(d\cdot\tau)-g_1(\tau))^2}, \ \ 
\frac{d^3\cdot g_3(d\cdot\tau )}{(d\cdot g_1(d\cdot \tau)-g_1(\tau))^3}\right)
\end{equation}
}
where $g_i$'s are the Eisenstein series given in \eqref{emiliomestrado2018}. The fact that $(g_1,g_2,g_3)$ is a solution of the Ramanujan differential equation implies the same
statement for $(d\cdot g_1(d\cdot\tau), d^2g_2(d\cdot \tau ), d^3g_2(d\cdot \tau ))$. By our construction of the 
vector field \eqref{14March2018} in \S\ref{11may2018}, we conclude that the image
of the map \eqref{4apr2018} is tangent to the vector field $\v$ in \eqref{14March2018}.
The map \eqref{4apr2018-2}  for $A=A_d$ gives us the map $\mm$ after taking the quotient by $\Gamma_0(d)$, 
where $A_d$ is defined in \eqref{zoorekhar2018}. 
 The curve $S_0(d)$ is just the quotient of the three dimensional variety $V_d$ by $\BG$. 
The fact that these are the only algebraic solutions follows from Proposition \ref{twoyearsago}.
Chow theorem implies that the image of the map \eqref{4apr2018}  is an algebraic curve. 
However, this theorem does not give any information about the field of definition. 
In \S\ref{cobertura2018} we give the description of $S_0(d)$ as a curve over  $\Q$.

For $A_d$ defined in \eqref{zoorekhar2018}, the equivalent classes in the quotient 
$\Gamma\backslash\Gamma A_d\Gamma$
give the same map  \eqref{4apr2018}. Each class is represented uniquely by an element of the from
\begin{equation}
\label{fazendamat2018}
 d^{-\frac{1}{2}}\mat{a}{e}{0}{b},\ \ d=ab,\ a,b,e\in\N,\ \ 0\leq e\leq b-1, 
\end{equation}
but not all these matrices are in $\Gamma A_d\Gamma$.  The cardinality of $\Gamma\backslash\Gamma A_d\Gamma$ is $\psi(d)$
wheras  the cardinality of matrices in \eqref{fazendamat2018} is $\sigma_1(d)$. 
The map \eqref{4apr2018-2} with $A$ as in \eqref{fazendamat2018} is
\begin{equation}
\label{soumacorrida2018}
\tau\mapsto \left[g_2(\tau ):g_3(\tau): a^2b^{-2}g_2\left ( \frac{a\tau+e}{b}\right): -a^3b^{-3}g_3\left( \frac{a\tau+e}{b}\right) : 
g_1(\tau)-ab^{-1}g_1\left( \frac{a\tau+e}{b}\right)  \right ]
\end{equation}
This evaluated at $\tau=i\infty$ (equivalently $q=0$) is
$$
\left[12 :8: 12 (ab^{-1})^2: -8(ab^{-1})^3 : 1-(ab^{-1})\right ]= g([2:-2a^{-1}b])=g([a:-b]).
$$
\begin{rem}\rm
It seems  that for any decomposition $d=ab,\ a,b\in\N$ we have  $\varphi(\gcd(a,b))$ smooth branches of 
$S_0(d)$ crossing the point $g([a:-b])$ in  the curve singularity of $\F(\v)$, where  $\varphi$ is the Euler's totient function.
This is compatible with the well-known fact that the cardinality of the set of cusps $\Gamma_0(d)\backslash\left(\Q\cup \{\infty\}\right)$ 
is   $\sum_{a\mid d}  \varphi( \gcd(a, \frac{d}{a}))$. 
\end{rem}   

{\it Proof of \ref{222}:}
The set $M_\R$ announced in the theorem is just the quotient of $M_\R$ in \eqref{desejos2018} by the diagonal action of $\BG$ and the 
first integral $B$ is just the first integral in \eqref{div2018} after taking the quotient. 


\begin{rem}\rm
The eigenvalues of the linear part of the isolated singularity $0\in\C^4$ of $\F$ are $2,3,-2,-3$ which is resonant and hence 
it does not lie in the Poincar\'e domain. Therefore, we may not be able to do  
a holomorphic change of coordinates in $(\C^4,0)$ such that $(\F,0)$ is equivalent
to its linear part which has the first integrals 
\begin{equation}
 \frac{(x_2+6x_3)^3}{x_3^2}={\rm const}_1,\ \ \frac{(y_2+6y_3)^3}{y_3^2}={\rm const}_2,
\end{equation}
see \cite{IlyashenkoYakovenko} Theorem 4.3.
\end{rem}
\begin{rem}\rm
The curve singularity of $\F$ intersects the weighted projective space at infinity $\{y_1=0\}\cong\Pn {2,3,2,3}$
at the point $g([1:-1])$. Therefore, an algebraic leaf $S_0(d)$ of $\F$ intersects the curve singularity at infinity if and only if 
$d$ is a square. 
\end{rem}

\section{Proof of Theorems \ref{16may2018}, \ref{saket2018}}
\label{cobertura2018}
Let us define
\begin{eqnarray*}
 P_f(x)&:=& \prod_{A\in \Gamma\backslash \Gamma A_d\Gamma }\left(x- f|_k^0 A\right),\ \ \ \ \ \ \ \ (f,k)=(g_2,4),(g_3,6),\\
 P_f(x)&:=& \prod_{A\in \Gamma\backslash \Gamma A_d\Gamma }\left(x-f+ f|_k^1 A\right ),\ \ (f,k)=(g_1,2), 
\end{eqnarray*}
where $\Gamma:=\SL 2\Z$ and the slash operators are defined in \eqref{marizek2018}. 
Using  Hecke operators one can prove that, $P_{g_i},\ \ i=1,2,3$ is a homogeneous 
polynomial degree $i\psi(d)$ in 
$$
\Q[x,g_2, g_3],\ \ \deg(x)=m,\ \   \deg(g_2)=2, \deg(g_3)=3, 
$$
This is classical for $i=2,3$, but less well-known for $i=1$. In this case we use the following equalities for $f=g_1$: 
\begin{eqnarray}
f|_2^1A|_2^0B &=& f|_2^1AB+c'(c'\tau+d')^{-1},\\
f|_2^0B & =&f+c'(c'\tau+d')^{-1}   \ \ \ \ \ \ \  \forall A, B=\mat{a'}{b'}{c'}{d'}\in\SL 2\C,   
\end{eqnarray}
and conclude that the coefficient $x^i$  of $P_{g_1}$ in $x$  is a  modular form of weight $2(\psi(d)-i)$  for $\Gamma$  and 
defined over $\Q$, 
and hence can be written as a polynomial of in $g_2,g_3$ with $\Q$ coefficients. For further details  
see Proposition 6 \cite{ho14-II}. Note that $\Gamma\backslash \Gamma A_d \Gamma$
is isomorphic to the fiber of the map 
$\Gamma\backslash \Mat_d(2,\Z)\to \Gamma\backslash \Mat_d(2,\Z)/\Gamma$ over 
over the matrix $\mat{d}{0}{0}{1}$, and in \cite{ho14-II} we have used the latter set in order to define $P_{g_i}$'s.  
In general, we can define $P_f$ for any quasi-modular form of weight $k$ and differential 
order $n$ for $\Gamma$. For examples of polynomials $P_{g_i}$, see \cite{ho14-II} page 440. 

The conclusion is that we have  three homogeneous polynomials  
$Q_{d,1}(y_1,x_2,x_3)$, $Q_{d,2}(y_2,x_2,x_3)$, $Q_{d,3}(y_3,x_2,x_3)$ of degrees respectively $\psi(d),2\psi(d),\ 3\psi(d)$ in 
the ring  $\Q[x_2,x_3,y_1,y_2,y_3]$,  $\deg(x_i)=\deg(y_i):=i$ such that 
\begin{align}
&Q_{d,1}(d\cdot g_1(d\cdot \tau)-g_1(\tau), g_2(\tau),g_3(\tau))=0,\\
& Q_{d,2}(d^2\cdot g_2(d\cdot \tau), g_2(\tau),g_3(\tau))=0,\\ 
&Q_{d,3}(d^3\cdot g_3(d\cdot \tau), g_2(\tau),g_3(\tau))=0
\end{align}
and hence they give three equations for $S_0(d)\subset \Pn w$. 

{\it Proof of Theorem \ref{16may2018}:} 
The polynomial $Q_{d,i}$ is a linear
combination of the monomials (\ref{8july2011}) as in \eqref{21may2018}.
We prove that $\det B_{d,i}$  restricted to $S_0(d)$ is identically zero. 
Since $\v$ is tangent to the curve $S_0(d)$, we know 
that for all $r\in \N_0$ we have   $\v^r(\sum_{j=1}^{m_{d,i}} c_j\alpha_{i,j})=\sum_{j=1}^{m_{d,i}} c_j\v^r_d(\alpha_{i,j})$ restricted to
$S_0(d)$ is zero. This in turn implies that the matrix $B_{d,i}$ restricted  to points of $S_0(d)$ 
has non-zero kernel and so its determinant restricted to $V_d$ is zero. 
The last part of our proof is a slight generalization to higher dimensions of an argument  due to J. V. Pereira, see  
\cite{Pereira2001} Proposition 1, page 1390. \qed

{\it Proof of Theorem \ref{saket2018} part 2:} 
Once we know a point $p\in S_0(d)$ then we can repeat the proof  of Theorem \ref{16may2018} and conclude that $B_{d,i}(p) C_{d,i}=0$. 

{\it Proof of Theorem \ref{saket2018} part 1:} 
Let us write our  elliptic curves $E_1:=E_{0,t_2,t_3}$ and $E_2:=E_{0,s_2,s_3}$ in the notation of the three parameter Weierstrass format \eqref{khodaya-1}.
The equality \eqref{omidpiano2018} can be written in the following format 
$$
 \begin{bmatrix}
  f^*\frac{dx}{y}, &
  f^*\frac{xdx}{y}  
 \end{bmatrix}=
\begin{bmatrix}
  \frac{dx}{y}, &
  \frac{xdx}{y}  
 \end{bmatrix}
 \g\cdot A_d,
 $$
 where $A_d$ is the matrix \eqref{zoorekhar2018} and $\g$ is in the algebraic group $\BG$ defined in \eqref{2nov10}. 
 This follows, for instance, form \cite{ho14-II} Proposition 1.
 Define $s_1:=0$ and $s_2,s_3$ as before and redefine
$$
(t_1,t_2,t_2):= (0,t_2,t_3)\bullet \g=(k'k^{-1},t_2k^{-4},t_3 k^{-6}),
$$
where we have used the action of the algebraic group $\BG$ in \eqref{2nov10}.  
The point $(t,s)\in\T\times\T$ lies in the algebraic set $V_d$ defined in  Remark 
\ref{edemgurilandiaflamengo2018}. After taking the quotient by the diagonal action of $\BG$ on $\T\times \T$ we get the point \eqref{26may2018} of $S_0(d)$.

\section{Computing at a cusp}
We can also state a theorem similar to Theorem \ref{saket2018} without the input of an 
isogeny and using a point in the curve singularity of $\F(\v)$ 
which corresponds to cusp points of $S_0(d)$. 
Note that at these points the vector field $\v$ vanishes and the matrix $B_{d,i}$
evaluated there has all its lines equal to zero except the first one. 
Therefore, we have to use second 
order approximation of modular curves at these points.
\begin{prop}
\label{23may2018}
Let $d=ab$ with $a,b\in \N,\ \ b<a,\ r:=\frac{a}{b}$ and $p:=g([r:-1])$. 
Let also $C_{d,i}$ be the $m_{d,i}\times 1$ matrix with entries $c_{i,j},\ j=1,2,\ldots,m_{d,i}$ which are coefficients of $Q_{d,i}$ in \eqref{21may2018}.
We have 
\begin{equation}
\label{omidagitado2018}
\left( \frac{\partial B_{d,i}}{\partial x_2}(p)\cdot (6-5r)(1-r) + \frac{\partial B_{d,i}}{\partial x_3}(p)\cdot (7r-6)+ 
 \frac{\partial B_{d,i}}{\partial y_2}(p)\cdot r^2(1-r) - \frac{\partial B_{d,i}}{\partial x_3}(p)\cdot r^3 \right) C_{d,i}=0
\end{equation}
\end{prop}
\begin{proof}
Since  the image $S_0(d)$ of the 
map $\mm $ in \eqref{4apr2018} is tangent to $\v$ and $Q_{d,i}$ restricted to it vanishes, 
we conlude that the pull-back   of $B_{d,i}$ by $\mm$ satisfies   $B_{d,i}(\mm(\tau))C_{d, i}=0$. The theorem follows from derivating this equality 
with respect to a variable $Q$  and then setting $Q=0$ which we explain it below. 

The map \eqref{soumacorrida2018} is just a different parametrization of $S_0(d)$ using the action of $\Gamma$ on $\uhp$.
Let us write $r=\frac{a'}{b'}$ with $\gcd(a',b')=1$. 
The $g_i$'s in this map have Fourier expansions in terms of $Q:=e^{\frac{2\pi i}{b'}}$. We have 
$g_i(\tau)=*+*Q^{b'}+\cdots$  and $g_i(r\tau +\frac{e}{b})=*+*Q^{a'}+\cdots $, where $*$'s are constants. Since $b'<a'$, we need only the constant
term of  $g_i(r\tau +\frac{e}{b})$. If we write $g_i=a_i(1+b_iq+\cdots )$ we get
$$
\mm(\tau)=g([r:-1])+
$$
$$
\left(
\frac{a_2(-b_2r+b_2-2b_1)}{a_1^2(1-r)^3}, \ \frac{a_3(-b_3r+b_3-3b_1)}{a_1^3(1-r)^4}, 
\frac{-2r^2a_2b_1}{a_1^2(1-r)^3},\ \frac{3r^3a_3b_1}{a_1^3(1-r)^4}
\right)Q+\cdots
$$
where $\cdots$ means higher order terms. Multiplying the coefficient of $Q$ with $-\frac{a_1^3(1-r)^4}{3a_3b_1}$  we get the desired constants in $r$ 
which are used in \eqref{omidagitado2018}.
Note that $(a_1,a_2,a_3)=(1,12,8)$ (up to $\frac{2\pi i}{12}$ factors which do not affect this 
computation) and  $(b_1,b_2,b_3)=(-24,240,-504)$. \qed
\end{proof}
A computer implementation of the matrix $B_{d,i}(p)$ for $p:=g([d:-1])$  and small values 
of $d$ shows that the rank of the matrix $B_{d,i}(p)$ is much below $m_{d,i}-1$, and hence, 
$C_{d,i}$ is not uniquely determined by the equality \eqref{omidagitado2018}. 
\href
{http://w3.impa.br/~hossein/WikiHossein/files/Singular%20Codes/2018-06-EllipticModularFoliationII.txt 
}
{For instance, for $d=2,3,4,5$} and $i=2$ we have $m_{d,i}=5,7,12,12$ and the rank of $B_{d,i}$ is respectively 
                                              $1,3,3,3$.


\def\Ha{{\rm H}}
\section{Self join of Halphen differential equation}
\label{12may2018}
We can carry out the self join of foliations introduced in \S\ref{11may2018} for the 
Halphen differential equation:
\begin{equation}
\label{tatlis}
\Ha:\left \{ \begin{array}{l}
\dot t_1=
(1-\alpha_1)(t_1t_2+t_1t_3-t_2t_3)+\alpha_1 t_1^2\\
\dot t_2 =
(1-\alpha_2)(t_2t_1+t_2t_3-t_1t_3)+\alpha_2 t_2^2\\
\dot t_3 =
(1-\alpha_3)(t_3t_1+t_3t_2-t_1t_2)+\alpha_3 t_3^2
\end{array} \right., 
\end{equation}
see \cite{ha81-1}, 
with $\alpha_i\in\C\cup\{\infty\}$ (if for instance $\alpha_1=\infty$ then the first row is replaced with 
$-t_1t_2-t_1t_3+t_2t_3+t_1^2$).
It turns out that if we consider two copies of $\Ha$ in $(t_1,t_2,t_3)$
and $(s_1,s_2,s_3)$ variables, but with the same parameters $\alpha_1,\alpha_2,\alpha_3$,  
and compute the derivation of the new variables 
$$
(x_1,x_2, y_1,y_2):= \left (\frac{t_2-t_1}{s_1-t_1}, \ \ \frac{t_3-t_1}{s_1-t_1},\ \ \frac{s_2-s_1}{t_1-s_1}, \ \ \frac{s_3-s_2}{t_1-s_1}\right )
$$
and divide the result over $s_1-t_1$, then we arrive at:
\begin{equation}
\label{DupHal2018}
\left\{
\begin{array}{l}
\dot x_1= \hspace{.1in} x_1\left( \alpha_2 x_1+(2-\alpha_2-\alpha_1)x_2-(1-\alpha_1)(x_1x_2-y_1y_2)-1  \right) \\
\dot x_2=  \hspace{.1in}   x_2\left( \alpha_3 x_2+(2-\alpha_3-\alpha_1)x_1-(1-\alpha_1)(x_1x_2-y_1y_2)-1  \right) \\
\dot y_1=- y_1\left( \alpha_2 y_1+(2-\alpha_2-\alpha_1)y_2-(1-\alpha_1)(y_1y_2-x_1x_2)-1  \right)  \\
\dot y_2= -  y_2\left( \alpha_3 y_2+(2-\alpha_3-\alpha_1)y_1-(1-\alpha_1)(y_1y_2-x_1x_2)-1  \right) 
\end{array}\right. 
\end{equation}
The Halphen vector field for $\alpha_1=\alpha_2=\alpha_3=0$ is a pull-back of the Ramanujan vector field with a degree $6$  
map, see \cite{ho14} page 335, and we have a similar relation between 
\eqref{DupHal2018} and $\v$ in this case. Therefore,  in Theorem \ref{jorgeduque2018} we have classified all algebraic solutions of \eqref{DupHal2018} with 
$\alpha=0$  in the Zariski open set $x_1x_2(x_1-x_2)\not=0,  y_1y_2(y_1-y_2)\not=0$. Some important ingredients of the proof of Theorem \ref{jorgeduque2018}
in the case of \eqref{DupHal2018} have been worked out
in \cite{ho07-1}. This includes the relation of \eqref{tatlis} and the period map, and 
the explicit computation of the stabilizer $\Gamma$ of a solution of \eqref{tatlis} as an explicit subgroup of $\SL 2\C$. 
The most critical part of the generalization would be to classify the set 
\begin{equation}
 C(\Gamma):=\left\{A\in \SL 2\C \Big| \ \ \ \ |\Gamma\backslash \Gamma A\Gamma|<\infty \ \ \ \ \ \right \}
\end{equation}
which we need it in Proposition \ref{twoyearsago}. 
Since we have actions of $\Gamma$ from both the left and right on $C(\Gamma)$, we are actually interested to classify 
the  double quotient $\Gamma \backslash C(\Gamma)/\Gamma$, which for $\Gamma=\SL 2\Z$ is isomorphic to $\N$ through $d\mapsto A_d$, where 
$A_d$ is the matrix in \eqref{zoorekhar2018}. 
In a personal communication A. Salehi Golsefidy 
pointed out that $C(\Gamma)$ contains the commensurability group, and even for this we do not know much beyond their Zariski-closure.
For a particular examples of $\alpha$, the 
group $\Gamma$ is the subgroup  of $\SL 2\R$ generated by 
\begin{equation}\label{trigen}\gamma_1=\left(\begin{matrix}2\cos(\frac{\pi}{m_1})&1\cr -1&0\end{matrix}\right)\,,
\gamma_2=\left(\begin{matrix}0&1\cr -1&2\cos(\frac{\pi}{m_2})\end{matrix}\right)\,,
\gamma_3=\left(\begin{matrix}1&2\cos(\frac{\pi}{m_1})+2\cos(\frac{\pi}{m_2})\cr 0&1\end{matrix}\right)
\end{equation}
for some $m_1,m_2\in \N,\ \frac{1}{m_1}+\frac{1}{m_2}<1$. This is a triangle group of type $(m_1,m_2,\infty)$. 
The generalization of Eisenstein series in this case are done in \cite{hokh2}. In this case 
it seems  that $C(\Gamma)$ is an infinite enumerable set if and only if $\Gamma$ is arithmetic. 
Note that we have a finite number of arithmetic triangle groups which are classified by Takeuchi in \cite{tak77}. 

\section{Picard's curious example}
\label{BolsoHaddad2018}
Let $\T$ be the moduli space of triples $(E,P,\omega)$, where $E$ is an elliptic curve over $\C$  and by definition it comes together with a 
point $O$, $P\not=O$ is another point in $E$ and
$\omega$ is a meromorphic differential $1$-form in $E$ with poles only at $O$ and $P$ and with the order of pole equal to 
one at both points. Moreover, the residue of $\omega$ at $P$ and $O$ is respectively $+1$ and $-1$. 
\begin{prop}
We have 
$$
\T\simeq \Pn {1,2,3,4}\backslash\{ [s:a:b:c]\in\Pn {1,2,3,4} \big| \Delta=0 \},
$$
where $\Delta:=27(-b^2+4a^ 3-ca )^ 2-c^3$, and the universal family over $\T$ is given by 
\begin{eqnarray} \label{bolsonaro2018}
E=E_{a,b,c} &:&  y^2=4x^3-cx+b^2-4a^ 3+ca, \\ \nonumber
\omega=\omega_{s,a,b} &:=& \frac{1}{2}\frac{y+b}{x-a}\frac{dx}{y}+s\frac{dx}{y},\ \  P=(a,b).
\end{eqnarray}
\end{prop}
\begin{proof}
 We choose Weierstrass coordinates $x,y$ on $E$. These are rational functions on $E$ with pole of order 
 $2$ and $3$ at $O$, respectively. In this way we can write $E$ in the Weierstrass format $E_{c,{\check c}}: y^2=4x^3-cx-{\check c}$ with 
 $\Delta:=27{\check c}^2-c^3\not=0$. In these coordinates we write $P=(a,b)$ and it follows that
 any triple of the moduli space $\T$ is isomorphic to a triple in \eqref{bolsonaro2018} for some 
 $(s,a,b,c)\in \C^4$. Note that  ${\check c}=4a^3-ca-b^2$ and so it can be discarded.    
 For $k\in\C^*$ we have 
 \begin{eqnarray*}
  & & f:  E_{k^ {-4}c,k^ {-6}{\check c}} \simeq  E_{c,{\check c}},\\
  & & f(x,y)=(k^2x,k^ 3y),\\
  & & f_*\omega_{k^ {-1}s, k^ {-2}a,k^ {-3}b}=\omega_{s,a,b}. 
 \end{eqnarray*}
 and so $(s, a,b,c)$ and $(k^ {-1}s,  k^ {-2}a,k^ {-3}b, k^ {-4}c)$ represents the same point in $\T$.
\end{proof}
Let us consider the affine chart $s=1$ for the moduli space $\T$. 
 Let $\delta\in H_1(E_{a,b,c},\Z)$ be a continuous family of cycles. 
 \href
 { 
http://w3.impa.br/~hossein/WikiHossein/files/Singular%20Codes/2018-10-PicardCuriousDifferentialEquation.txt
 }
 {A simple, but long, calculus computation gives us the following:}
 \begin{equation}
 \label{haddad2018}
 d \int_{\delta}\left(\frac{1}{2}\frac{y+b}{x-a}\frac{dx}{y}+\frac{dx}{y}\right) =\frac{\alpha_1}{\Delta} \cdot  
 \int_\delta\frac{xdx}{y}+  \frac{\alpha_2}{\Delta} \cdot   \int_\delta\frac{ dx}{y},
 \end{equation}
 where $d$ is the differential of holomorphic functions in $(a,b,c)\in\C^3$, 
 $\alpha_i=\alpha_{1i}da+\alpha_{2i}db+\alpha_{3i}dc,\ \ i=1,2$ and $\alpha_{ij}$'s are given in 
 {\tiny
$$
\alpha:=
\left(
\begin{array}{*{2}{c}}
3c^{2}-36ca^{2}+45cab-108a^{3}b+27b^{3} & -\frac{1}{2}\left( 9c^{2}a+3c^{2}b-144ca^{3}+54ca^{2}b+9cb^{2}+432a^{5}-216a^{4}b-108a^{2}b^{2}+54ab^{3}\right) \\
(2c^{2}-30ca^{2}+6cb+72a^{4}-18ab^{2}) & -(2c^{2}a-30ca^{3}+9cab+3cb^{2}+72a^{5}-36a^{3}b-18a^{2}b^{2}+9b^{3}) \\
-\frac{1}{2}\left(3ca+3cb-36a^{3}+9b^{2}\right) & \frac{1}{4}\left( c^{2}-18ca^{2}+9cab+72a^{4}-36a^{3}b-18ab^{2}+9b^{3}\right)
\end{array}
\right).
$$
}
This matrix has rank two and the vector field 
\begin{equation}
\label{10oct2018}
\v:=
(2c-24a^{2}+6ab+6b)\frac{\partial}{\partial a}
-(3c-36a^{2}+36ab-9b^{2})\frac{\partial}{\partial b}+
(12ca+12cb-144a^{3}+36b^{2})\frac{\partial}{\partial c}
\end{equation}
is in the kernel of $\alpha_1$ and $\alpha_2$ and generates it.
This implies that along the the solutions of the vector field $\v$ in $\T$,  the integral in the left hand side of \eqref{haddad2018}
is constant for all continuous family of cycles. 

The foliation $\F(\v)$ in $\T$ has infinite number of algebraic leaves $S_1(N), N=2,3,\cdots$. The leaf $S_1(N)$
parameterizes the triples $(E,\frac{1}{N}\frac{df_N}{f_N},P)$, where $P$ is a torsion point of order $N$ and $f_N$
is a rational function in $E$ with ${\rm div}(f_N)=N\cdot (P-O)$. 
One can give a parametrization of $S_1(N)$ by modular forms as follows. We consider the complex torus $E:=\frac{\C}{\Z\tau+\Z}$ and its 
embedding in $\Pn 2$ using $z\mapsto [\wp(z,\tau): \wp'(z,\tau):1]$, where $\wp(z,\tau)$ is the Weierstrass $\wp$ function and its
derivation means with respect to $z$. We also consider the torsion point $P=\frac{1}{N}$ in $E$. 
The following function 
$$
F_N(\tau):=\frac{1}{N}\frac{f_N'}{f_N}-\frac{1}{2}\frac{\wp'(z,\tau)+ \wp'(\frac{1}{N},\tau)}{ \wp(z,\tau)-\wp(\frac{1}{N},\tau)  }
$$
is holomorphic on the torus and hence it is independent of $z$. Here, $f_N(z)$ is a double periodic function in $z$ with period $1$ and $\tau$
and it has a zero (resp. pole) of order $N$ at $z=\frac{1}{N}$ (resp. $z=0$) and $'$ means derivation with respect to $z$.  
The compactification of the curve $S_1(N)$ 
in $\Pn {1,2,3,4}$ is birational to the modular curve $X_1(N):=\Gamma_1(N)\backslash \uhp^*$,  where
$$
\Gamma_1(N):=\left\{\mat{a_1}{a_2}{a_3}{a_4}\in \SL 2\Z\Bigg| a_3\equiv 0\ a_1\equiv a_4\equiv 1,\ \ ({\rm mod }\ \  N)\right\}.
$$
Such a birational map is given by 
$$
\Gamma_1(N)\backslash \uhp^*\to \Pn {1,2,3,4},\ \tau\mapsto 
\left[F_N(\tau):\wp\left(\frac{1}{N},\tau\right): \wp'\left(\frac{1}{N},\tau\right): 60G_4(\tau)\right].
$$
The four functions involved in the above parameterization are modular forms for $\Gamma_1(N)$.
The precise comparision of our Picard's differential equation given by $\v$ in \eqref{10oct2018} 
and the Picard's differential  equation in \cite{Picard1889} pages 298-299  is left to the reader.
For the line bundle $L:=\O(P-O)$ on $E=E_{a,b,c}$  with its global meromorphic  section $s$ such that ${\rm div}(s)=P-O$, 
we can associate the holomorphic connection $\nabla: L\to \Omega_E\otimes L,\ s\mapsto \omega_{s,a,b}\otimes s$. Isomonodromic
defomrations of $(E,\nabla)$ is the same as deformations with constant integrals in the left hand side of 
\eqref{haddad2018}. This and   \cite{Loray2016} Corollary 2 and 3 have been the starting point of our reformulation
of Picard's example.


\section{Final comments}
In the present text we have avoided the arithmetic of modular curves which is a vast territory of research with fruitful applications
such as arithmetic modularity theorem. There are few topics which would fit perfectly into this article and we mention them briefly. 
Using geometric Hecke operators for (quasi) modular forms, see for instance \cite{ho14-II} page 432, 
one can prove that $Q_{d,i}$ have coefficients in $\Z[\frac{1}{6}]$
\href{http://w3.impa.br/~hossein/WikiHossein/files/Singular%20Codes/2018-06-Katz-Grothendieck.txt
}
{which might be used for mod $p$ study of modular curves.} 
For arithmetic purposes such as those in \cite{DeligneRapoport1973} it would be essential to classify the bad primes of 
$S_0(d)$ and the  classification of fibers over bad primes. One might use a desingularization of $\F$ along the curve singularity so 
that one gets smooth models 
of modular curves.
We started our article from a probelm posed by Darboux in \cite{Darboux1978} and elaborated a counterexample to this probelm. Surprisingly, one main ingredient
of this counterexample is the  differential equation \eqref{tatlis} with $\alpha=0$  which Darboux himself derived in the article \cite{da78}; both articles being 
published in 1878.

\def\cprime{$'$} \def\cprime{$'$} \def\cprime{$'$} \def\cprime{$'$}


\end{document}